\documentclass[a4paper,12pt]{article}
\usepackage{mathrsfs}
\usepackage{}
\usepackage[top=2.0cm,bottom=2.0cm,left=2.0cm,right=2.0cm]{geometry}
\usepackage{amssymb}
\usepackage{graphicx}
\usepackage{amsmath,amsthm,amssymb,lineno}
\usepackage{latexsym}
\usepackage{epstopdf}
\usepackage{setspace}
\usepackage{amsmath,color}
\usepackage{graphicx,booktabs,multirow}
\usepackage{latexsym, tabularx,shapepar}
\usepackage[all,2cell,dvips]{xy} \UseAllTwocells \SilentMatrices
\usepackage{appendix}
\usepackage{longtable}
\usepackage{cite}
\usepackage{CJK}
\usepackage{indentfirst}
\usepackage{array}
\usepackage[colorlinks,
           linkcolor=blue,
           anchorcolor=blue,
           citecolor=blue
           ]{hyperref}
\usepackage{amsthm}
\allowdisplaybreaks[4]
\graphicspath{{figures/}}

\newtheorem{theorem}{Theorem}[section]
\newtheorem{lemma}[theorem]{Lemma}

\newtheorem{corollary}[theorem]{\rm\bfseries Corollary}

\begin{document}
%\linenumbers

\title{ The nullity of the net Laplacian matrix of a signed graph
}
\author{Zhuang Xiong \thanks{Corresponding author: zhuangxiong@hunnu.edu.cn}\\
{\small Key Laboratory of High Performance Computing and Stochastic Mathematics (Ministry of Education),} \\
{\small College of Mathematics and Statistics, Hunan Normal University, Changsha, Hunan 410081,  China} \\
}

\date{}
\maketitle
\begin{abstract}

  Let $\Gamma = (G, \sigma)$ be a signed graph, where $G = (V(G),E(G))$ is an (unsigned) graph, called the underlying graph. The net Laplacian matrix of $\Gamma$ is defined as $L^{\pm}(\Gamma) = D^{\pm}(\Gamma) - A(\Gamma)$, where $D^{\pm}(\Gamma)$ and $A(\Gamma)$ are the diagonal matrix of net-degrees and the adjacency matrix of $\Gamma$, respectively.
  The nullity of $L^{\pm}(\Gamma)$, written as $ \eta (L^{\pm} (\Gamma))$, is the multiplicity of 0 as an eigenvalue of $L^{\pm}(\Gamma)$. In this paper, we focus our attention on the nullity of the net Laplacian matrix of a connected  signed graph $\Gamma$ and prove that $1 \leq  \eta (L^{\pm} (\Gamma)) \leq min\{ \beta(\Gamma) + 1, |V(\Gamma)| - 1 \}$, where $\beta(\Gamma) = |E(\Gamma)| - |V(\Gamma)| + 1$ is the cyclomatic number of $\Gamma$. The connected signed graphs with nullity $|V(\Gamma)| - 1$ are completely determined. Moreover, we characterize the signed cactus graphs with nullity $1$ or $\beta(\Gamma) + 1$.
\end{abstract}
\baselineskip=0.25in

\section{ Introduction}\label{sec:into}
A signed graph $\Gamma$ of order $n$ is a pair $(G,\sigma)$, where $G = (V(G),E(G))$ is an $n$-vertices (unsigned) graph with vertex set $V(G)$ and edge set $E(G)$, called the underlying graph, and $\sigma: E(G) \rightarrow \{-1, +1\}$ is the sign function.
For a vertex $v$ of $\Gamma$, the positive degree $d_{\Gamma}^+(v)$ of $v$ in $\Gamma$ is the number of positive neighbours of $v$ (i.e., those adjacent to $v$ by a positive edge). In the similar way, we define the negative degree $d_{\Gamma}^-(v)$.
The net-degree of $v$ in $\Gamma$ is defined as $d_{\Gamma}^{\pm}(v) = d_{\Gamma}^{+}(v) - d_{\Gamma}^{-}(v)$.

\indent Given a matrix $M$, the spectrum of $M$ is denoted by $Spec(M) = \{\lambda_1(M)^{k_1}, \cdots, \lambda_i(M)^{k_i} \}$, where the superscripts denote the multiplicities of corresponding eigenvalues. Throught this paper, the eigenvalues of any matrix are all arranged in non-increasing order. Denote by $r(M)$ and $\eta(M)$ the rank and the nullity of $M$, respectively. The adjacency matrix $A(\Gamma)$ of a signed graph $\Gamma$ is obtained from the adjacency matrix of the underlying graph by reversing the sign of all 1s corresponding to negative edges. The net Laplacian matrix of $\Gamma$ is defined as $L^{\pm}(\Gamma) = D^{\pm}(\Gamma) - A(\Gamma)$, where $D^{\pm}(\Gamma)$ is the diagonal matrix of net-degrees. \\
\indent Recently the nullity of the adjacency matrix of a simple graph has received a lot of attention. Collatz and Sinogowitz \cite{collatz1957spektren} posed the problem of characterizing all singular graphs. The nullity of a graph is a classical topic in spectrum theory of graphs due to its applications in chemistry. In the H\"{u}ckel molecular orbital model, if $\eta(A(G)) > 0$ for the molecular graph $G$, then the corresponding chemical compound is highly reactive and unstable, or nonexistent (see \cite{atkins2006physical} or \cite{cvetkovic1972the}). In studing the above problem, some attentions are attracted to bound the nullity of a graph by using some of the structure parameters, such as the order, the maximum degree, the number of pendent vertices, and the number of cyclomatic number of the graph, etc (see \cite{chang2011rank4,chang2012rank5,cheng2007on,cvetkovic1972the,
omidi2009bipart,wang2020proof,wang2022graphs}).\\
\indent The nullity of the adjacency matrix of a signed graph also has been widely studied (see \cite{chen2022rank,fan2013note,fan2014nullity,liu2014further,lu2018the} and reference therein). This problem is closely related to the minimum rank problem of symmetric matrices whose patterns are described by graphs \cite{fallat2007the}. Here we consider this problem with respect to the net Laplacian matrix. A significance of the spectrum of the net Laplacian matrix in control theory was recognized in \cite{gao2018equitable}. The same topic is studied in \cite{stanic2020net} from a graph theoretic insight. The advantages of use of the net Laplacian matrix instead of the Laplacian matrix (in study of signed graphs) is investigated in \cite{stanic2020on}. Very recently, Mallik \cite{mallik2022matrix} introduced a new oriented incidence matrix of a signed graph, by which the matrix tree theorem of the net Laplacian matrix of a signed graph is given. In this paper, we investigate the nullity of the net Laplacian matrix of a connected signed graph, which relies on the study of the characteristic polynomial of the net Laplacain matrix of this signed graph. In 1982, Chaiken \cite{chaiken1982com} gave a combinatorial proof of the all minors matrix tree theorem. In 2016, Buslov \cite{buslove2016coefficients} proposed an alternative proof based on the straightforward computation of the minors of incidence matrices and on revealing a connection of them with forests. The above two papers established a way for computing any coefficient of the characteristic polynomial of the Laplace matrix of a weighted digraph. Here we rewrite it in a form of the net Laplacian matrix of a signed graph $\Gamma$. The proof can be obtained directly from \cite[All minors matrix tree theorem]{chaiken1982com} or \cite[Theorem 2]{buslove2016coefficients}. Denote by  $\mathcal{F}^k(\Gamma)$ the set of all spanning $k$-component forests of $\Gamma$. For $F^k(\Gamma) \in \mathcal{F}^k(\Gamma)$, $a(F^{k}(\Gamma)) = n_1 \ldots n_k$, where $n_i$ is the number of the vertices of $i$-component of $F^k(\Gamma)$.
    \begin{theorem}\label{tho:cof}
      Let $$ P_{ L^{\pm}(\Gamma) } (x)=  \det(xI - L^{\pm}(\Gamma)) = \sum_{k=0}^{n} c_k x^k $$ be the characteristic polynomial of the net Laplacian matrix of a signed graph $\Gamma$. Then
    $$c_k = (-1)^{n-k}(\sum_{F^k(\Gamma) \in \mathcal{F}^k(\Gamma)} a(F^{k}(\Gamma))\sigma(F^k(\Gamma))).$$
    \end{theorem}

 To ease language, in the rest of this paper, we abbreviate the spectrum, the nullity and the rank of $L^{\pm}(\Gamma)$ as the spectrum, the nullity and the rank of $\Gamma$ and denote them by $Spec(\Gamma)$, $\eta(\Gamma)$ and $r(\Gamma)$, respectively. Obviously, $r(\Gamma) + \eta(\Gamma) = n$ if the order of $\Gamma$ is $n$. Note that $L^{\pm}(\Gamma)$ is a symmetric matrix and the sum of each row in it is equal to $0$. Thus, $\eta(\Gamma) \geq 1$ holds for any signed graph $\Gamma$ and $\eta(\Gamma) = 1$ if and only if $c_1 \neq 0$, where the coefficient $c_1$ of the linear term of $P_{L^{\pm}(\Gamma)}(x)$ is $$(-1)^{n-1}(\sum_{F^1(\Gamma) \in \mathcal{F}^1(\Gamma)} a(F^{1}(\Gamma))\sigma(F^1(\Gamma))) = (-1)^{n-1} n (\sum_{F^1(\Gamma) \in \mathcal{F}^1(\Gamma)} \sigma(F^1(\Gamma))) .$$
 Note that $\eta(\Gamma) = \eta(-\Gamma)$, where $-\Gamma$ is obtained by reversing the sign of each edge in $\Gamma$. A signed cactus graph $\Gamma$ is a signed graph whose underlying graph is a cactus graph. Recall that a cactus graph is a connected graph in which any two cycles have no edge in common. Equivalently, it is a connected graph in which any two cycles have most one vertex in common. The main result is as follows.
\begin{theorem}\label{tho:nullupp}
 Let $\Gamma$ be a connected signed graph of order $n$ $( n \geq 2)$ with cyclomatic number $\beta(\Gamma)$. \\
  \indent $(i)$ $1 \leq  \eta (\Gamma) \leq min\{ \beta(\Gamma) + 1, n - 1 \}$,\\
  \indent $(ii)$ $\eta(\Gamma) = n - 1$ if and only if $n$ is an even number and $\Gamma = K_{\frac{n}{2}} \triangledown^{-} K_{\frac{n}{2}}$ or $-(K_{\frac{n}{2}} \triangledown^{-} K_{\frac{n}{2}})$,\\
   where $K_{\frac{n}{2}}$ is the $\frac{n}{2}$-vertices signed complete graph with all positive edges and $K_{\frac{n}{2}} \triangledown^{-} K_{\frac{n}{2}}$ is obtained by adding all possible negative edges between vertices of one $K_{\frac{n}{2}}$ and vertices of another one.\\
   \indent Moreover, if $\Gamma$ is a signed cactus graph, then\\
  \indent $(iii)$ $\eta (\Gamma) = 1$ if and only if $m^+(C) \neq m^-(C)$ for any cycle $C$ of $\Gamma$, \\
  \indent $(iv)$ $\eta (\Gamma) = \beta(\Gamma) + 1$ if and only if $m^+(C) = m^-(C)$  for any cycle $C$ of $\Gamma$,\\
  where $m^+(C)$ and $m^-(C)$ are the numbers of the positive and negative edges of $C$, respectively.
\end{theorem}

Note that the inequalities $1 \leq  \eta (\Gamma) \leq  \beta(\Gamma) + 1$ are also established by Ge and Liu in \cite[Theorem 6.17]{ge2022sym}. In this paper, we give a short proof of these inequalities. As a by-product of the study  of the spectra of signed complete graphs, Ou, Hou, and Xiong  proved in \cite[Corollay 2.9]{ou2021net} that, for an $n$-vertices signed complete graph $(K_n,\sigma)$, $\eta((K_n,\sigma)) = n-1$ if and only if $(K_n,\sigma)$ is $K_{\frac{n}{2}} \triangledown^{-} K_{\frac{n}{2}}$ or $-(K_{\frac{n}{2}} \triangledown^{-} K_{\frac{n}{2}})$. Here without computing the spectra we extend this result from signed complete graphs to connected signed graphs.\\
\indent The remainder of this paper is organized as follows. Some lemmas are introduced in Section \ref{sec:pre}. In Section \ref{sec:proof} we give the proof of Theorem \ref{tho:nullupp}.

\section{Preliminaries}\label{sec:pre}
We first recall some notations not defined in Section \ref{sec:into}.
Let $\Gamma$ be a signed graph with vertex set $V(\Gamma)$ and edge set $E(\Gamma)$.
A subgraph $H$ of $\Gamma$ is a signed graph such that $V(H) \subseteq V(\Gamma)$, $E(H) \subseteq E(\Gamma)$ and the edge set $E(H)$ preserving the signs in $\Gamma$. Further, $H$ is called an induced subgraph of $\Gamma$ if for $\forall$ $u,v \in V(H)$, $u,v$ are adjacent in $H$ if and only if they are adjacent in $\Gamma$.
The sign of a subgraph $H$ of $\Gamma$ is defined as $\sigma(H) = \prod_{e \in E(H)}\sigma(e)$ and the numbers of positive and negative edges of $H$ are denoted by $m^+(H)$ and $m^-(H)$, respectively.
If $V_1 \subseteq V(\Gamma)$, we denote by $\Gamma[V_1]$ the induced subgraph of $\Gamma$ with vertex set $V_1$, and  denote by $\Gamma - V_1$ the induced subgraph of $\Gamma$ with vertex set $V(\Gamma) \setminus V_1$, i.e., $\Gamma - V_1 = \Gamma[V(\Gamma) \setminus V_1]$. We simplify $\Gamma-V_1$ as $\Gamma-v$ when $V_1 = \{ v\}$.
For an induced subgraph $H$ of $\Gamma$ and a vertex subset $V_1 \subset V(\Gamma)$ outside $H$, denote by $H + V_1$ the induced subgraph of $\Gamma$ with vertex set $V(H) \cup V_1$. Sometimes we use the notation $\Gamma- H$ instead of $\Gamma-V(H)$ if $H$ is an induced subgraph of $\Gamma$.
For an edge subset $E_1 \subseteq E(\Gamma)$, we denote by $\Gamma - E_1$ the signed graph with the same vertex set as $\Gamma$ and with edge set $E(\Gamma) \setminus E_1$. We also abbreviate $\Gamma - E_1$ as $\Gamma - e$ when $E_1 = \{ e\}$.  An edge $e$ (resp., a vertex $v$) is called a cut edge (resp., a cut vertex) if $\Gamma - {e}$ (resp., $\Gamma - v$) has more connected components than $\Gamma$. A signed graph $\Gamma$ with a cut vertex $w$ can be regard as a coalescence $\Gamma_1 \cdot \Gamma_2$ of two signed graphs $\Gamma_1$ and $\Gamma_2$, obtained from $\Gamma_1 \: \dot \cup \: \Gamma_2$ by identifying a vertex $u$ of $\Gamma_1$ with a vertex $v$ of $\Gamma_2$. (Formally, $V(\Gamma_1  \cdot  \Gamma_2) = V(\Gamma_1-u) \: \dot \cup \: V(\Gamma_2-v) \: \dot \cup \: \{ w\}$ with two vertices in $\Gamma_1 \cdot \Gamma_2$ adjacent if they are adjacent in $\Gamma_1$ or $\Gamma_2$, or if one is $w$ and the other is a neighbour of $u$ in $\Gamma_1$ or a neighbour of $v$ in $\Gamma_2$). The cyclomatic number of a connected signed graph $\Gamma$, denoted by $\beta(\Gamma)$, is defined as $\beta(\Gamma) = |E(\Gamma)| - |V(\Gamma)| + 1$. A connected signed graph $\Gamma$ is called a  signed tree (resp., a signed unicyclic graph; a signed bicyclic graph) if $\beta(\Gamma) = 0$ (resp., $\beta(\Gamma) = 1; \beta(\Gamma) = 2)$.\\
  \indent Next we present some preliminary results which are useful in the sequel. The following lemma is clear.

 \begin{lemma}\label{lem:empty} Let $\Gamma$ be a signed graph of order $n$.\\
  \indent $(1)$ If $\Gamma$ = $\Gamma_{1} \; \dot \cup \cdots \dot \cup \; \Gamma_{t}$, where $\Gamma_{1}, \cdots ,\Gamma_{t}$ are all the connected components of $\Gamma$, then $\eta(\Gamma) = \sum _{i = 1}^{t} \eta(\Gamma_{i})$.\\
  \indent $(2)$ $\eta(\Gamma) = n$ if and only if $\Gamma$ has no edges.
  \end{lemma}

  Next we will introduce an analog for Interlacing Theorem (c.f. \cite[Theorem 2.1]{hammers1995interlacing}) of the net Laplacian matrix of a signed graph in respect of edges. For this reason we need the lemma below, which is known as Courant-Weyl inequalities.
\begin{lemma}\cite[Theorem 1.3.15]{cvetkovic2010introduction}\label{lem:Coutant_Wel}
Let $A$ and $B$ be $n \times n$ Hermitian matrices. Then
\begin{gather*}
\lambda_{i}(A + B) \leq \lambda_{j}(A) + \lambda_{i-j+1}(B) \qquad (1 \leq j \leq i \leq n), \\
\lambda_{i}(A + B) \geq \lambda_{j}(A) + \lambda_{i-j+n}(B) \qquad (1 \leq i \leq j \leq n).
\end{gather*}
\end{lemma}

   Note that we cannot invoke an analog for Interlacing Theorem of the net Laplacian matrix when we delete vertices, because a principal submatrix of $L^{\pm}$ is not the net Laplacian matrix of the corresponding induced subgraph. However we do have an analog for Interlacing Theorem when we delete an edge:
   \begin{lemma}\label{lem:interlacing}
    Let $\Gamma$ be a signed graph of order $n$. If $e = uv$ is an edge of $\Gamma$ and $H = \Gamma - e$, then
   \begin{gather*}
\lambda_{1}(\Gamma) \geq \lambda_{1}(H) \geq \cdots \geq \lambda_{n}(\Gamma) \geq \lambda_{n}(H), \qquad if \; \sigma(e) = +1. \\
\lambda_{1}(H) \geq \lambda_{1}(\Gamma) \geq \cdots \geq \lambda_{n}(H) \geq \lambda_{n}(\Gamma), \qquad if \; \sigma(e) = -1.
   \end{gather*}
   \end{lemma}
    \begin{proof}
   We can write $L^{\pm}(\Gamma)$ as $ L^{\pm}(H) + Q$, where
\[Q = \bordermatrix{
  &   & u & v &  & \cr
  & \textbf{0} & \textbf{0} & \textbf{0} & \textbf{0} \cr
u & \textbf{0}^{\top} & \sigma(e) & -\sigma(e) & \textbf{0}^{\top} \cr
v & \textbf{0}^{\top} & -\sigma(e) & \sigma(e) & \textbf{0}^{\top} \cr
  & \textbf{0} & \textbf{0} & \textbf{0} & \textbf{0} \cr
}.
\]
Note that the spectrum of $Q$ is $\{2^{1}, 0^{n-1} \}$ if $\sigma(e) = +1$, and is $\{0^{n-1}, (-2)^{1} \}$ if $\sigma(e) = -1$. Then by Lemma \ref{lem:Coutant_Wel}, we obtain the result.
    \end{proof}

    We end this section with a direct consequence of Lemma \ref{lem:interlacing}.
    \begin{corollary}\label{col:nulltiy} Let $\Gamma$ be a signed graph. If we delete an edge $e$ of $\Gamma$, then $\eta(\Gamma)-1 \leq \eta(\Gamma - e) \leq \eta(\Gamma) + 1$.
    \end{corollary}

\section{ Proof of Theorem \ref{tho:nullupp}}\label{sec:proof}
 This section is devoted to the proof of Theorem \ref{tho:nullupp}.
 We need some helpful lemmas as preparations. First we study the nullities of signed trees and signed unicyclic graphs in following two lemmas.
 In fact, it has been proved in \cite[Theorem 2.2]{fieler1975eigenvector} that, for an $n$-vertices signed tree $T$, the numbers of positive, negative, and zero eigenvalues of $T$ are $m^+(T)$, $m^-(T)$, and $n-m^+(T)-m^-(T) = 1$, respectively.
 Here we also give a different proof for self-contained.

\begin{lemma}\label{lem:tree}
Let $T$ be a signed tree of order $n$. Then $\eta(T) = 1$.
\end{lemma}
\begin{proof}
By Theorem \ref{tho:cof} and $\eta(T) \geq 1$, it is sufficient to prove  $c_{1} \neq 0$ in the case $\Gamma = T$. The result follows from $c_{1} = (-1)^{n-1} n \sum_{F^1(T) \in \mathcal{F}^1(T)} \sigma(F^1(T)) = (-1)^{n-1}n \cdot \sigma(T) \neq 0$.
\end{proof}

\begin{lemma}\label{lem:unicyc}
Let $U$ be a signed unicyclic graph of order $n$ with the unique signed cycle $C$. Then
\begin{equation*}
 \eta(U) =
  \left\{
    \begin{array}{cc}
     1, & if \: m^+(C) \neq m^-(C), \\
     2 , & otherwise.  \\
    \end{array}
  \right.
 \end{equation*}
\end{lemma}
\begin{proof}

As we have shown in the above, we need to prove that, in the case $\Gamma = U$, $c_1 \neq 0$ if $m^+(C) \neq m^-(C)$ and $ c_1 = 0$ otherwise. Since
\begin{equation*}
\sum_{F^1(U) \in \mathcal{F}^1(U)}\sigma(F^1(U)) = \sum_{e \in E(C)} \sigma(U - e)
      = \sigma(U) \sum_{e \in E(C)}\sigma(e),
\end{equation*}
where the first equality follows from that an $1$-component spanning forest (i.e. a spanning tree) of $U$ is obtained by deleting an edge of $C$, so we have $c_1 \neq 0$ and $\eta(U) = 1$ if $m^+(C) \neq m^-(C)$. If $m^+(C) = m^-(C)$, on the one hand, $c_1 = 0$ and so $\eta(U) \geq 2$. On the other hand, by Corollary \ref{col:nulltiy} and Lemma \ref{lem:tree} we have $\eta(U) \leq \eta(U - e) + 1 = 2 $, where $e\in E(C)$. This completes the proof.
\end{proof}

We next study how $\eta(\Gamma)$ changes when we delete a cut edge or a cut vertex from a signed graph.
\begin{lemma}\label{lem:cut_edge}
Let $\Gamma$ be a signed graph with a cut edge $e = uv$ and $\Gamma-e = \Gamma_1 \; \dot  \cup \; \Gamma_2$, where $\Gamma_1$ and $\Gamma_2$ are two induced subgraphs of $\Gamma-e$ containing $u$ and $v$, respectively. Then $\eta(\Gamma) \geq \eta(\Gamma_1) + \eta(\Gamma_2) - 1$.
\end{lemma}
\begin{proof}
By arranging the vertices of $\Gamma$ appropriately we can write $L^{\pm}(\Gamma)$ as
\[ \bordermatrix{
  &   & u & v &  & \cr
  & B & \alpha & \textbf{0} & \textbf{0} \cr
u & \alpha^{\top} & d^{\pm}_{\Gamma}(u) & -\sigma(e) & \textbf{0}^{\top} \cr
v & \textbf{0}^{\top} & -\sigma(e) & d^{\pm}_{\Gamma}(v) & \gamma^{\top} \cr
  & \textbf{0} & \textbf{0} & \gamma & D \cr
}.
\]
Adding all other rows and columns to the row and column indexed by $u$, respectively, we obtain a matrix
\begin{equation*}
\centering
{\begin{matrix}
M = \begin{pmatrix}
  B & \textbf{0} & \textbf{0} & \textbf{0} \cr
  \textbf{0}^{\top} & 0 & 0 & \textbf{0}^{\top} \cr
  \textbf{0}^{\top} & 0 & d^{\pm}_{\Gamma}(v) & \gamma^{\top} \cr
  \textbf{0} & \textbf{0} & \gamma & D \cr
\end{pmatrix} = \begin{pmatrix}
  B & \textbf{0} & \textbf{0} & \textbf{0} \cr
  \textbf{0}^{\top} & 0 & 0 & \textbf{0}^{\top} \cr
  \textbf{0}^{\top} & 0 & d^{\pm}_{\Gamma}(v)-\sigma(e) & \gamma^{\top} \cr
  \textbf{0} & \textbf{0} & \gamma & D \cr
\end{pmatrix}
 + \begin{pmatrix}
  \textbf{0} & \textbf{0} & \textbf{0} & \textbf{0} \cr
  \textbf{0}^{\top} & 0 & 0 & \textbf{0}^{\top} \cr
  \textbf{0}^{\top} & 0 & \sigma(e) & \textbf{0}^{\top} \cr
  \textbf{0} & \textbf{0} & \textbf{0} & \textbf{0} \cr
\end{pmatrix}.
\end{matrix}}
\end{equation*}
Note that the rank of $L^{\pm}(\Gamma)$ is the same with that of $M$, $r(B) = r (\Gamma_1)$ and $\sigma(e) \neq 0$. Thus, $r(\Gamma) = r(M) \leq r(\Gamma_1) + r(\Gamma_2) + 1$, which means that $\eta(\Gamma) \geq \eta(\Gamma_1) + \eta(\Gamma_2) - 1$.
\end{proof}

\begin{lemma}\label{lem:ident}
If $\Gamma$ is a signed graph with a cut vertex $w$ and $\Gamma = \Gamma_1 \cdot \Gamma_2$. Then $\eta(\Gamma) = \eta(\Gamma_1) + \eta(\Gamma_2) - 1$.
\end{lemma}
\begin{proof}
By arranging the vertices of $\Gamma$ appropriately we can write $L^{\pm}(\Gamma)$ as
\[ \bordermatrix{
  &   & w &  &  \cr
  & B_1 & \alpha_1 & \textbf{0} \cr
w & \alpha^{\top}_1 & d_{\Gamma}^{\pm}(w) & \alpha_2^{\top}  \cr
  & \textbf{0} & \alpha_2 & B_2 \cr
}.
\]
Adding all other rows and columns to the row and column indexed by $w$, respectively, we obtain a matrix
\[
\begin{pmatrix}
 B_1 & \textbf{0} & \textbf{0} \\
 \textbf{0}^{\top} & 0 & \textbf{0}^{\top}  \\
 \textbf{0} & \textbf{0} & B_2
\end{pmatrix}.
\]

\noindent Note that $r(B_1) = r(\Gamma_1)$ and $r(B_2) = r(\Gamma_2)$. Thus, we have $r(\Gamma) = r(B_1) + r(B_2) = r(\Gamma_1) + r(\Gamma_2)$, which means that $\eta(\Gamma) = n + 1 - r(\Gamma_1) - r(\Gamma_2) - 1 = \eta(\Gamma_1) + \eta(\Gamma_2) - 1$.
\end{proof}

With the help of Lemmas \ref{lem:tree} and \ref{lem:ident}, we present a corollary below, which can simplify the structures of the signed graphs we consider.

\begin{corollary}\label{col:ident}
If $\Gamma = \Gamma_1 \cdot T$ is a signed graph and $T$ is a signed tree. Then $\eta(\Gamma) = \eta(\Gamma_1)$.
\end{corollary}

It is clear from Corollary \ref{col:ident} that, when we consider the nullity of a signed graph $\Gamma = \Gamma_1 \cdot T_{1} \cdot T_{2} \cdot \cdots \cdot T_k$, where $T_{1}, T_{2}, \cdots, T_k$ are signed trees, we only need to determine the nullity of $\Gamma_1$. So in the following proof, we can always assume that the signed graph with no pendent signed trees.

\indent Now we are in a position to prove the main result.

\begin{proof}[Proof of Theorem \ref{tho:nullupp}]
For convenience we abbreviate $\beta(\Gamma)$ as $\beta$ and choose edges $e_1, \cdots , e_\beta$ from $\Gamma$ such that $T =\Gamma - \{e_1, \cdots, e_\beta \}$ is a signed tree. Then by Corollary \ref{col:nulltiy}, we have
\begin{gather*}
\eta(\Gamma) - 1 \leq \eta(\Gamma-e_1),\\
\eta(\Gamma) - 2  \leq \eta(\Gamma-e_1) - 1   \leq \eta(\Gamma- \{e_1, e_2\}),\\
\vdots\\
\eta(\Gamma) - \beta  \leq \eta(\Gamma- \{e_1, \cdots, e_{\beta-1} \} ) - 1   \leq \eta(\Gamma- \{e_1, \cdots ,e_\beta \}).
\tag{3.1}\label{inequ:main}
\end{gather*}
From the inequalities of (\ref{inequ:main}), we obtain $\eta(\Gamma)  \leq \eta(T) + \beta = 1 + \beta$.
By Lemma \ref{lem:empty} $(2)$, there does not exist any $n$-vertices connected signed graphs with nullity $n$, which means $\eta(\Gamma)\leq n-1$. So we obtain the inequalities of $(i)$.\\
\indent When $\eta(\Gamma) = n - 1$ (i.e., $r(\Gamma) = 1$), if there exists an element $l_{ij}$ of $L^{\pm}(\Gamma)$ equal to $0$, then all the elements in the same row and column with it are $0$s, which is contrary to the fact $\Gamma$ is connected. Thus, $\Gamma$ is a signed complete graph and so $d_{\Gamma}^{\pm}(v_i) \in \{ -1,1 \}$, where $d^{\pm}_{\Gamma}(v_i)$ is the net-degree of any vertex $v_i$ in $\Gamma$. Without loss of generality, assume that there exists a vertex $v_1$ in $\Gamma$ with $d^{\pm}_{\Gamma}(v_1) = 1$ (if not, we can consider the signed graph $-\Gamma$ instead of $\Gamma$), which means that $d_{\Gamma}^{+}(v_1) = n/2$ and $d_{\Gamma}^{-}(v_1) = n/2 - 1$. Thus $n$ is an even number. By arranging the vertices of $\Gamma$ appropriately, we denote by $v_2, \cdots, v_{\frac{n}{2}}$ the all negative neighbours of $v_1$. So far we have determined the elements in the row and column indexed by $v_1$ in $L^{\pm}(\Gamma)$. Using the condition $r(\Gamma) = 1$, we can write $L^{\pm}(\Gamma)$ as
\[ \bordermatrix{
  & v_1  & v_2 & \cdots & v_{\frac{n}{2}} & v_{\frac{n}{2}+1} & \cdots & v_{n-1}  & v_{n} \cr
v_1  & 1 & 1 & \cdots & 1 & -1 & \cdots & -1  & -1 \cr
v_2 & 1 & 1 & \cdots & 1 & -1 & \cdots & -1  & -1 \cr
\vdots & \vdots & \vdots & \ddots & \vdots & \vdots & \ddots & \vdots  & \vdots \cr
v_{\frac{n}{2}}  & 1 & 1 & \cdots & 1 & -1 & \cdots & -1   & -1 \cr
v_{\frac{n}{2}+1} & -1 & -1 & \cdots & -1 & 1 & \cdots & 1   & 1 \cr
\vdots & \vdots & \vdots & \ddots & \vdots & \vdots & \ddots & \vdots  & \vdots \cr
v_{n-1} & -1 & -1 & \cdots & -1 & 1 & \cdots & 1  & 1 \cr
v_{n} & -1 & -1 & \cdots & -1 & 1 & \cdots & 1  & 1 \cr
}.
\]
So $\Gamma = K_{\frac{n}{2}} \triangledown^{-} K_{\frac{n}{2}}$ or $ -(K _{\frac{n}{2}} \triangledown^{-} K_{\frac{n}{2}})$. This shows the necessity of $(ii)$ and the sufficiency is obvious.

\indent Assume that $\Gamma$ is a signed cactus graph which contains no pendent trees in the rest of proof. \\
\indent For $(iii)$, the assertion follows from Lemmas \ref{lem:tree} and \ref{lem:unicyc} when $\beta = 0,1$. Therefore, we divide the proof into two cases in which Case $2$ will rely on the induction on the cyclomatic number of $\Gamma$ and Case $1$ follows from a direct computation. \\
\indent\textbf{Case 1:} $\Gamma$ has no cut edges.\\
Assume that $\Gamma$ is obtained by a series of coalescence of the cycles $C_1, \cdots, C_\beta$. By Lemma \ref{lem:ident} we have $\eta(\Gamma) = \eta(C_1) + \cdots + \eta(C_\beta) - \beta + 1 = \beta - \beta + 1 = 1$, where the second equality is from Lemma \ref{lem:unicyc} and $m^+(C_i) \neq m^-(C_i)$ for each $i = 1, \cdots, \beta$.\\
\indent In the latter case, we assume that the assertion holds for signed cactus graphs with cyclomatic number at most $\beta - 1$. Recall that $\eta(\Gamma) = 1$ if and only if the coefficient $c_1$ of the linear term of $P_{L^{\pm}(\Gamma)}(x)$ does not equal to zero. \\
\indent\textbf{Case 2:} $\Gamma$ has a cut edge $e = uv$.\\
Denote the signed graph $\Gamma-e$ as $\Gamma_1 \; \dot \cup \;\Gamma_2$ such that $1 \leq \beta(\Gamma_1), \beta(\Gamma_2) \leq \beta - 1$. Since each cycle contained in $\Gamma_1$ or $\Gamma_2$ also has distinct numbers of positive and negative edges, we have
$\eta(\Gamma_1) = 1 = \eta(\Gamma_2)$ by the inductive hypothesis. Thus,
\begin{gather*}
\sum_{F^1(\Gamma_1) \in \mathcal{F}^1(\Gamma_1)} \sigma(F^1(\Gamma_1)) \neq 0,  \sum_{F^1(\Gamma_2) \in \mathcal{F}^1(\Gamma_2)} \sigma(F^1(\Gamma_2)) \neq 0.
\end{gather*}
Since $e$ is a cut edge, any spanning tree of $\Gamma$ must contain edge $e$. Then we have
$$
\mathcal{F}^1(\Gamma) = \bigcup_{F^1(\Gamma_1) \in \mathcal{F}^1(\Gamma_1)} \bigcup_{F^1(\Gamma_2) \in \mathcal{F}^1(\Gamma_2)} \{ F^1(\Gamma_1) \cup F^1(\Gamma_2) + e \},
$$
where $F^1(\Gamma_1) \cup F^1(\Gamma_2) + e$ is obtained by adding an edge $e$ to $F^1(\Gamma_1) \cup F^1(\Gamma_2)$,
and so
\begin{gather*}
\sum_{F^1(\Gamma) \in \mathcal{F}^1(\Gamma)} \sigma(F^1(\Gamma)) = (\sum_{F^1(\Gamma_1) \in \mathcal{F}^1(\Gamma_1)} \sigma(F^1(\Gamma_1))) \cdot \sigma(e) \cdot (\sum_{F^1(\Gamma_2) \in \mathcal{F}^1(\Gamma_2)} \sigma(F^1(\Gamma_2))).
\end{gather*}
Thus, $c_1 \neq 0$ if and only if both $(\sum_{F^1(\Gamma_1) \in \mathcal{F}^1(\Gamma_1)} \sigma(F^1(\Gamma_1)))$ and $(\sum_{F^1(\Gamma_2) \in \mathcal{F}^1(\Gamma_2)} \sigma(F^1(\Gamma_2)))$ do not equal to zero, that is, $c_1 \neq 0$ if and only if both $\Gamma_1$ and $\Gamma_2$ are signed cactus graphs in which each cycle has distinct numbers of positive and negative edges. The cycles of $\Gamma$ consist of that of $\Gamma_1$ and $\Gamma_2$, so by inductive hypothesis we complete the proof in this case.\\
\indent If $\eta(\Gamma)  = \beta + 1$, suppose to the contrary that there exists a signed cycle $C$ in $\Gamma$
such that $m^+(C) \neq m^{-}(C)$. We take $\beta-1$ edges $e_1, \cdots, e_{\beta-1}$ in $\Gamma$ such  that $U = \Gamma - \{e_1,\cdots,e_{\beta-1} \}$ is a signed unicyclic graph with the unique cycle $C$. By Corollary \ref{col:nulltiy} and Lemma \ref{lem:unicyc} we obtain $2 = \eta(\Gamma) - \beta + 1 \leq \eta(\Gamma - \{e_1, \cdots, e_{\beta-1} \}) = 1$, a contradiction.  This shows the necessity of $(iv)$. \\
\indent In what follows, we will prove $\eta(\Gamma) = \beta + 1$, where $\Gamma$ is a signed cactus graph in which any cycle $C$ has $m^+(C) = m^-(C)$. We also divide our proof into two parts according to two cases.\\
\indent\textbf{Case 1:} $\Gamma$ has no cut edges.\\
Assume that $C_1, \cdots, C_\beta$ are all cycles in $\Gamma$. In fact, $\Gamma$ can be obtained by a series of coalescence of these cycles. So from Lemma \ref{lem:ident} we have $\eta(\Gamma) = \eta(C_1) + \cdots + \eta({C_\beta}) - \beta + 1 = 2\beta - \beta + 1 = \beta + 1$, where the second equality is from Lemma \ref{lem:unicyc} and $m^+(C_i) = m^-(C_i)$ for each $i = 1, \cdots, \beta$.\\
\indent In the remaining case, we proceed by induction on the cyclomatic number of $\Gamma$. Assume that the assertion holds for all signed cactus graphs with cyclomatic number at most $\beta - 1$ and let $\Gamma$ be a signed cactus graph with cyclomatic number $\beta$.\\
\indent\textbf{Case 2:} $\Gamma$ has a cut edge $e = uv$.\\
Denote the signed graph $\Gamma-e$ as $\Gamma_1 \; \dot \cup \; \Gamma_2$ such that $1 \leq \beta(\Gamma_1), \beta(\Gamma_2) \leq \beta-1$, where $\beta(\Gamma_1)$ and $\beta(\Gamma_2)$ are the cyclomatic numbers of $\Gamma_1$ and $\Gamma_2$, respectively. Note that each cycle contained in $\Gamma_1$ or $\Gamma_2$ also has the equal number of positive and negative edges and $\beta(\Gamma)= \beta(\Gamma_1) + \beta(\Gamma_2)$. By Lemma \ref{lem:cut_edge} and inductive hypothesis we have
\begin{gather*}
\eta(\Gamma) \geq \eta(\Gamma_1) + \eta(\Gamma_2) - 1 = \beta(\Gamma_1) + 1 + \beta(\Gamma_2) + 1 - 1 = \beta(\Gamma) + 1.
\end{gather*}
Combining this with $\eta(\Gamma) \leq \beta+1$ we obtain the conclusion in this case.
\end{proof}

\begin{figure}[htbp]\centering\hspace{-1.2cm}
\scalebox{0.4}{\includegraphics[width=8cm]{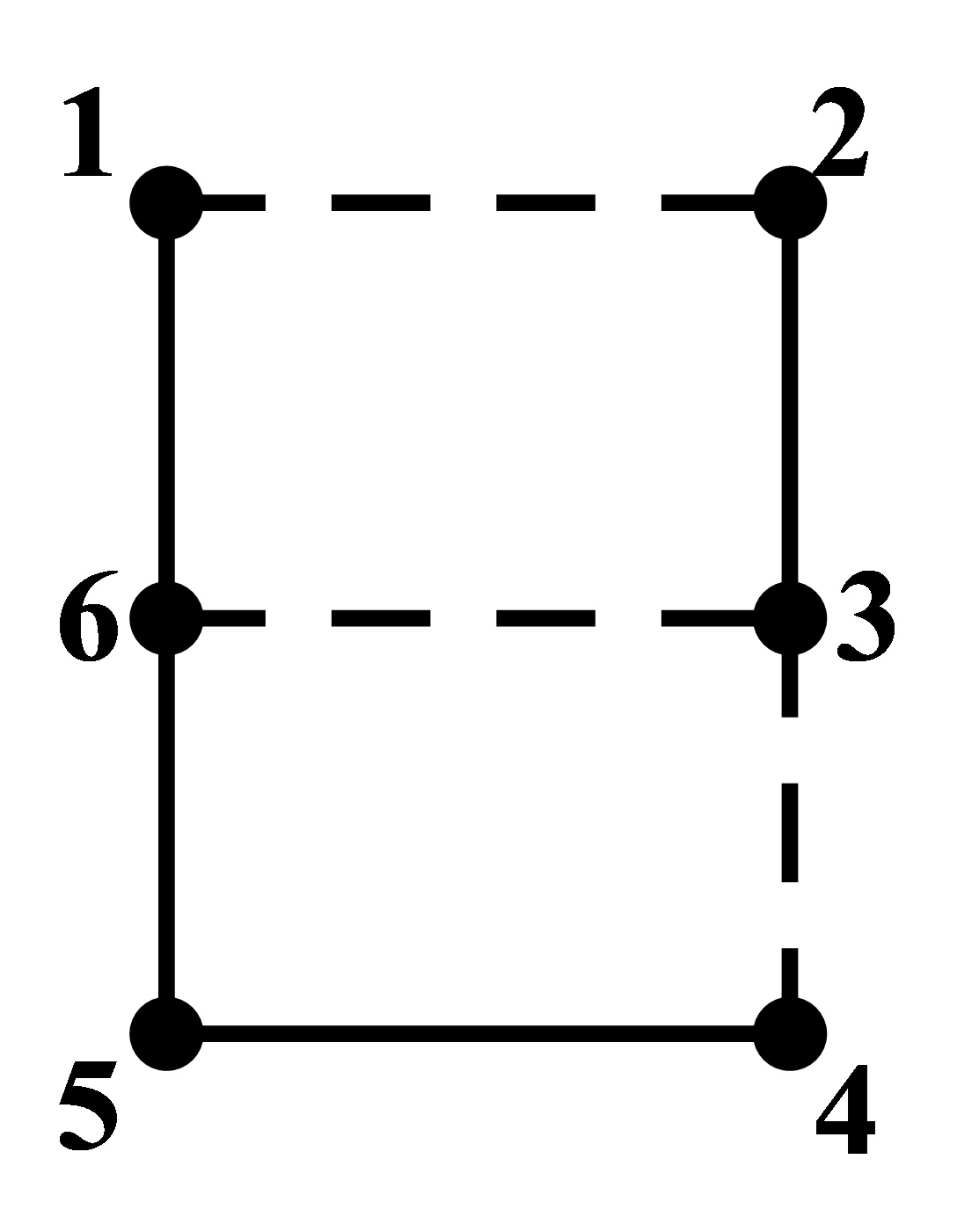}}

Figure 1. A signed graph of nullity $1$.

\end{figure}

With the help of Theorem \ref{tho:nullupp}, for a signed cactus graph $\Gamma$, we study how $\eta(\Gamma)$ changes when we delete an edge of them.
\begin{corollary}
Let $\Gamma$ be a signed cactus graph with cyclomatic number $\beta$ and $e$ be an edge of any cycle in $\Gamma$.\\
\indent $(1)$ If $m^+(C) = m^-(C)$ for any cycle $C$ of $\Gamma$, then $\eta(\Gamma-e) = \eta(\Gamma) - 1.$\\
\indent $(2)$ If $m^+(C) \neq m^-(C)$ for any cycle $C$ of $\Gamma$, then $\eta(\Gamma-e) = \eta(\Gamma) = 1.$
\end{corollary}
\begin{proof} As we have shown in the proof of Theorem \ref{tho:nullupp}, all the inequalities of (\ref{inequ:main}) become equalities when $\eta(\Gamma) = \beta + 1$, which leads to $(1)$.\\
\indent For $(2)$, Since $\Gamma-e$ is also a signed cactus graph in which any cycle has distinct numbers of positive and negative edges, then by Theorem \ref{tho:nullupp} we have $\eta(\Gamma-e) = 1 = \eta(\Gamma)$.
\end{proof}

\noindent \textbf{Concluding remarks.} Note, however, that we only consider signed cactus graphs achieved the nullity $0$ or $\beta(\Gamma) + 1$. For signed graphs in which two cycles have common edges, the discussion seems more complicated. For example, the nullity of the signed graph shown in Figure $1$
is $1$, but there exist two cycles of this signed graph such that the numbers of positive and negative edges on both of them are equal. Thus, the problem of characterizing all the signed graphs with nullity $1$ or $\beta(\Gamma) + 1$ are need to be resolved in further study.\\

\end{document}